\documentclass[11pt, usletter,reqno]{amsart}
\usepackage{chngcntr}

\newif\ifdraft\draftfalse
\newif\ifcite
\newif\ifblow

\usepackage{amsfonts,amssymb, amsmath}
\usepackage{oldgerm}
\usepackage{amscd}

\usepackage[hyperindex]{hyperref}  
\ifdraft\ifcite\usepackage{showkeys}\else\usepackage[notcite,notref]{showkeys}\fi\fi

\usepackage{amsmath, amsthm, amssymb, eucal, amscd, xypic}
\usepackage{mathpazo}

\makeatletter

\numberwithin{equation}{section}
\theoremstyle{plain}
\newtheorem{proposition}[equation]{Proposition}

\newtheorem{theorem}[equation]{Theorem}

\newtheorem{lemma}[equation]{Lemma}
\newtheorem{conjecture}[equation]{Conjecture}

\newtheorem{query}[equation]{Question}
\newtheorem{corollary}[equation]{Corollary}

\theoremstyle{remark}

\theoremstyle{definition}

\theoremstyle{remark}

\newtheorem{remark}[equation]{Remark}

\DeclareMathOperator\Prym{\mathrm{Prym}}
\DeclareMathOperator\cPrym{\overline{\Prym}}

\DeclareMathOperator\IH{\mathrm{IH}}

\newcommand\arr{\ifinner \to\else\longrightarrow\fi}
\newcommand\arrto{\ifinner\mapsto\else\longmapsto\fi}

 \newcommand\rH{\mathrm{H}}

\newcommand\IC{\mathop{\mathrm{IC}}\nolimits}

\newcommand\checkp{\check{\mathbb{P}}}

\begin{document}
\bibliographystyle{plain}
\title[Perverse obstructions]{Perverse obstructions
to flat regular compactifications}
\author{Patrick Brosnan}

\begin{abstract}
  Suppose $\pi:W\to S$ is a smooth, proper morphism over a variety $S$
  contained as a Zariski open subset in a smooth, complex variety
  $\bar S$.  The goal of this note is to consider the question of when
  $\pi$ admits a regular, flat compactification.  In other words, when
  does there exists a flat, proper morphism
  $\bar\pi:\overline{W}\to\bar S$ extending $\pi$ with $\overline{W}$
  regular?  One interesting recent example of this occurs in the
  preprint~\cite{lsv} of Laza, Sacc\`a and Voisin where 
  $\pi$ is a family of abelian $5$-folds over a Zariski open subset
  $S$ of $\bar S=\mathbb{P}^5$.  In that paper, the authors construct $\overline{W}$
  using the theory of compactified Prym varieties and show that it is
  a holomorphic symplectic manifold (deformation equivalent to
  O'Grady's $10$-dimensional example).

  In this note I observe that non-vanishing of the local intersection
  cohomology of $R^1\pi_*\mathbb{Q}$ in degree at least $2$ provides
  an obstruction to finding a $\bar\pi$.  Moreover, non-vanishing in
  degree $1$ provides an obstruction to finding a $\bar\pi$ with
  irreducible fibers.   Then I observe that, in some cases of interest,
results of Brylinski, Beilinson and Schnell can be used to compute
the intersection cohomology~\cite{BeilinsonRadon, BrylinskiRadon, SchRes}. 
I also give examples involving cubic $4$-folds
  motivated by~\cite{lsv} and ask a question about palindromicity
  of hyperplane sections. 
\end{abstract}
\maketitle

\section{Introduction}\label{intro}

Let $\bar S$ denote a smooth, quasi-projective, complex (irreducible)
variety of dimension $d$, and let $S$ denote a non-empty Zariski open
subset of $\bar S$.  Suppose $\pi: W\to S$ is a smooth, proper
morphism of relative dimension $n$.  I will call an irreducible,
regular scheme $\overline{W}$ equipped with a proper morphism $\bar\pi:\overline{W}\to \bar{S}$ a \emph{regular compactification} of $\pi$ if
\begin{enumerate}
\item $\overline{W}$ contains $W$ as a Zariski dense open subset;
\item the restriction of $\bar\pi$ to $W$ is $\pi$.  
\end{enumerate}

\begin{query}\label{q0}
  Under what conditions can we find a regular compactification
  $\bar\pi:\overline{W}\to \bar S$ of $\pi$ which is flat over $\bar S$.
  Also, under what conditions can we find a $\bar\pi:\overline{W}\to \bar S$
  as above with irreducible fibers?
\end{query}

My goal in this note is to write down some necessary conditions in
terms of local intersection cohomology.  For this, let $j:S\to\bar S$
denote the inclusion of $S$ in $\bar S$.  Pick an integer $k$
and set $\mathbf{L}=R^k\pi_*\mathbb{Q}$.  
Then the intersection complex $\IC\mathbf{L}$ 
is a polarizable Hodge module on $\bar S$ with underlying perverse
sheaf given by the intermediate extension of the underlying local system
$L$ to $\bar S$.
The intersection complex is also called the IC complex and is also
written as $j_{!*} \mathbf{L}[d]$.  The underlying perverse
sheaf is a complex of sheaves with
cohomology in the interval $[-d,0)$.  The local intersection
cohomology of $\mathbf{L}$ at a point $s\in\bar S$ is
\begin{equation}\label{licdef}
\IH^j_s\mathbf{L}:=H^{j-d}(\mathbf{L})_s.
\end{equation}
So $\IH^j_s\mathbf{L}$ is the $(j-d)^{\text{th}}$ cohomology of the
stalk of $\IC \mathbf{L}$ at $s$.  Clearly $\IH^j_s\mathbf{L}=0$
unless $j\in [0,d)$.  Moreover, $\IH^0_s \mathbf{L}$ is the space of
local invariants of $\mathbf{L}$ at $s$.  So,
$\IH^0_s\mathbf{L}$ is the fiber $\mathbf{L}_s$ for $s\in S$.  At points
$s\in \bar S$, $\IH^0_s\mathbf{L}=\Gamma(B\cap S, \mathbf{L})$ for a
sufficiently small ball $B$ in $\bar S$ containing $s$.

The following theorem, which I believe is a well-known
consequence of the decomposition theorem of Beilinson, Bernstein
and Deligne~\cite{BBD}, gives a way to obtain information about
possible compactifications $\bar\pi$ from the topology of $\pi$.  For the convenience of the reader
I will prove it in Section~\ref{proof}.

\begin{theorem}\label{gdf}
  Suppose $\bar\pi:\overline{W}\to \bar S$ is a regular compactification 
of $\pi$.  Then 
\begin{enumerate}
\item the complex 
$\oplus_{i} \IC(R^{n+i}\pi_*\mathbb{Q})[-i]$
includes in $R\bar\pi_*\mathbb{Q}[d+n]$ 
as a direct factor;
\item  for every integer $m$ and each point $s\in\bar S$,
  $\oplus_{j+k=m} \IH^{j}_s(R^k\pi_*\mathbb{Q})$ includes as a 
direct factor in the cohomology group $\rH^{m}(\bar W_s,\mathbb{Q})$ of the fiber of $\bar\pi$ over $s$.  
\item If the inclusion in (i) is an isomorphism, then so is the inclusion in (ii).   
\end{enumerate}
\end{theorem}

\begin{corollary}\label{CorOb}
If 
a flat, regular compactification $\bar\pi$ of $\pi$ exists, then, for all
$s\in \bar{S}$ and all integers $j,k$ with $j+k>2n$,  
$\IH^j_s(R^k\pi_*\mathbb{Q})=0$.  If the fibers of $\bar\pi$ are 
irreducible, then $\IH^0(R^{2n}\pi_*\mathbb{Q})=\mathbb{Q}$
and the groups $\IH^j_s(R^{2n-j}\pi_*\mathbb{Q})$  vanish for $j>0$.
\end{corollary}

\begin{proof}
If $\bar\pi$ is flat, then $\dim \overline{W}_s=n$ for all $s\in\bar{S}$.
So $\rH^{m}(\overline{W}_s,\mathbb{Q})=0$ for $m>2n$.  Thus Theorem~\ref{gdf} (ii)
implies the first statement.   If the fiber $\overline{W}_s$ is 
irreducible, then $\rH^{2n}\overline{W}_s=\mathbb{Q}$.  The
constant sheaf is a direct factor in $R^{2n}\pi_*\mathbb{Q}$.  So 
$\dim\IH^0_s(R^{2n}\pi_*\mathbb{Q})\geq 1$ for all $s\in\bar S $. 
The rest of Corollary~\ref{CorOb} is now immediate from Theorem~\ref{gdf}.
\end{proof}

In writing this note, I was mainly motivated by a recent preprint of
Laza, Sacc\`a and Voisin which concerns the situation where 
$\pi:A\to S$ is an abelian scheme of relative dimension $n$~\cite{lsv}.  In
this case, set $\mathbf{H}:= R^1\pi_*\mathbb{Q}(1)$.  It is a
polarized variation of Hodge structure of weight $-1$ on $S$, which,
is isomorphic to $R^{2n-1}\pi_*\mathbb{Q}(n)$ by Hard Lefschetz.  We
get the following.

\begin{corollary}\label{ihvan}
  Set $k=\max\{j: \IH^j_s\mathbf{H}\neq 0\}$.  Suppose a flat regular compactification
  $\bar\pi:\bar A\to \bar S$ of $\pi$ exists.  Then $k\leq 1$.   If $\bar A_s$ is
  irreducible, then $k=0$. 
\end{corollary}

\begin{proof}  This follows directly from Corollary~\ref{CorOb}
  applied to $R^{2n-1}\pi_*\mathbb{Q}$.   
\end{proof}

Suppose $X$ is a smooth, closed, $2m$-dimensional subvariety of
$P:=\mathbb{P}^N$ for some positive integers $m$ and $N$.  By cutting $X$
with hyperplanes, we get a family $\mathcal{X}\to P^{\vee}$  over
the dual projective space, which is smooth over a Zariski dense
open subset $U\subset P^{\vee}$.  (See \S\ref{Beil}.)  Set $n:=2m-1$ so that the
general member of the family $\mathcal{X}\to P^{\vee}$ is, by Bertini, a smooth
$n$-dimensional variety.  We get a variation of Hodge
structure $\mathbf{H}_{\mathbb{Z}}$ over $U$ such that the fiber over $H\in P^{\vee}$ 
is $\rH^n(X\cap H,\mathbb{Z}(m))$.   Let $J(\mathbf{H}_{\mathbb{Z}})\to U$ denote
the family of Griffiths intermediate Jacobians of $\mathbf{H}_{\mathbb{Z}}$.  In very
special cases, it turns out to be an abelian scheme.   Write $\mathbf{H}$ for the $\mathbb{Q}$-variation of Hodge structure obtained by tensoring 
$\mathbf{H}_{\mathbb{Z}}$ with $\mathbb{Q}$. 
In Section~\ref{Beil}, I will prove the following theorem (which,
along with Corollary~\ref{PalCor}, assumes the notation of the
preceding paragraph).

\begin{theorem}\label{BCor}
  Suppose that $\mathbf{H}$ is non-constant.   Let $H\in\checkp^N$ be a hyperplane
  and write $Y:=H\cap X$ for the hyperplane section.  Write
  $b_k Y:=\dim \rH^k(Y,\mathbb{Q})$ for the $k$-th Betti number. Then, for $k>0$,
  $$b_{n+k} Y - b_{n-k} Y = \dim \IH^k_H \mathbf{H}.$$
\end{theorem}

Call $Y$ \emph{palindromic} (resp. \emph{weakly palindromic})
if $b_{n+k} Y=b_{n-k} Y$ for all $k$ (resp. for all $k>1$).

\begin{corollary}\label{PalCor}
  Suppose that $\pi:J(\mathbf{H}_{\mathbb{Z}})\to U$ is a non-constant abelian scheme
  admitting a flat, regular compactification $\bar\pi:\bar J\to P^{\vee}$.
  Fix $H\in P^{\vee}$ and set $Y=X\cap H$.
  Then $Y$ is weakly palindromic.
If 
 the fiber of $\bar\pi$ over $H\in P^{\vee}$
  is irreducible, then  
  $Y$ is a palindromic  
\end{corollary}
\begin{proof}
  Since $R^1\pi_*\mathbb{Q}(1)\cong \mathbf{H}$, Corollary~\ref{PalCor} follows from 
Theorem~\ref{BCor} and Corollary~\ref{ihvan}.
\end{proof}

In~\cite{lsv}, the authors produce a flat, regular compactification $\bar A$ of a
family $\pi:A\to U$ of abelian $5$-folds over an open subset of $\mathbb{P}^5$.
In Section~\ref{exs}, I will
give examples where
Corollary~\ref{PalCor} can be used to rule out the existence of a
flat, regular compactification, or a flat, regular compactification with
irreducible fibers.   I will also give a consequence (Corollary~\ref{PalCor2})
of the main result
of~\cite{lsv} and state a conjecture (Conjecture~\ref{PalConj}) 
about palindromicity partially motivated by the results of~\cite{lsv}.

\subsection*{Acknowledgments}
This work was made possible by an NSF Focused Research Project on Hodge
theory and moduli held in collaboration with M.~Kerr, Laza, G.~Pearlstein
and C.~Robles. 
In fact, 
the note itself began as an email to R.~Laza
(dated August 22, 2016).  I thank the FRG members for
encouragement and useful conversations.

I also thank M.~Nori for giving me a lot of help with \S\ref{exs}, and
B.~Klingler for inviting me and Nori to Paris Diderot during the
Summer 2016.  I  thank J.~Achter for telling me about the tables
in M.~Rapoport's paper~\cite{RapDel}, and N.~Fakhruddin for advice on
how to improve the exposition.

The interaction between palindromicity and intersection cohomology
comes up in a similar way to the way it is used here
in my joint paper~\cite{BrCh} written with T.~Chow.  I thank Chow for
many conversations about the notion of palindromicity.

\section{Proof of Theorem~\ref{gdf}}\label{proof}

\begin{proof}
  Let $d$ denote the dimension of $\bar S$ and let $n$ denote the
  dimensions of the fibers of $\pi$.  By the decomposition theorems of
  Beilinson---Bernstein---Deligne~\cite{BBD} and Saito~\cite{MHP}, we
  have
  $R\bar\pi_*\mathbb{Q}[d+n]= \oplus_{i\in\mathbb{Z}}
  \mathcal{F}_i[-i]$ where the $\mathcal{F}_i$ are direct sums of
  intersection complexes coming from local systems on various strata.
  The restriction of $\mathcal{F}_i$ to $S$ is equal to
  $R^{n+i}\pi_*\mathbb{Q}[d]$.  So each $\mathcal{F}_i$ contains
  $\IC(R^{n+i}\pi_*\mathbb{Q})$ as a direct factor, and this implies
  Theorem~\ref{gdf} (i).  

  Let $\iota:\{s\}\to \bar S$ denote the inclusion of the point $s$.
  By proper base change, we have
  $\rH^m(\overline{W}_s,\mathbb{Q})=H^{m}(\iota^*R\bar\pi_*\mathbb{Q})
  =H^{m-d-n}(\iota^*R\bar\pi_*\mathbb{Q}[d+n])=\oplus_i
  H^{m-d-n}(\iota^*\mathcal{F}_i[-i])$.  This vector space contains as
  a direct factor the space
  \begin{align*}
\oplus_i H^{m-d-n}(\iota^*\IC(R^{n+i}\pi_*\mathbb{Q})[-i])
  &=\oplus_i H^{m-d-n-i}(\iota^*\IC(R^{n+i}\pi_*\mathbb{Q}))\\
  &= \oplus_i \IH^{m-n-i}(R^{n+i}\pi_*\mathbb{Q}))\\
  &= \oplus_{j+k=m} \IH^{j}(R^k\pi_*\mathbb{Q}).
  \end{align*}
  Moreover, if the inclusion in (i) is an isomorphism, the two spaces are equal.
  This proves (ii) and (iii).
\end{proof}

\section{Proof of Theorem~\ref{BCor}}\label{Beil}

Now we fix the notation from the introduction that $X$ is a smooth $2m$
dimensional closed subvariety of $P=\mathbb{P}^N$ and $P^{\vee}$ is the dual
projective space.   Let $\mathcal{X}:=\{(x,H)\in X\times P^{\vee}:
x\in H\}$ denote the incidence variety.   Write $q$ and $p$ for the
projections on the first and second factors respectively.  Then $q$
is a $\mathbb{P}^{N-1}$-bundle.  So $\mathcal{X}$ is smooth and irreducible
of dimensions $d_{\mathcal{X}}=n+N$ with $n=2m-1$.
On the other hand, the fiber of $p:\mathcal{X}\to P^{\vee}$ over a hyperplane
$H$ is the hyperplane section $Y_H:=H\cap X$.    Write $U$ for the locus
of hyperplanes $H$ such that $Y_H$ is smooth, and set $\mathcal{X}_U=p^{-1}(U)$.
Then the restriction of $p$ to $\mathcal{X}_U$ gives a smooth, proper morphism
$p_U:\mathcal{X}_U\to U$.

Set $\mathbf{H}:=R^{n}p_{U*}\mathbb{Q}(m)$.
This is a weight $-1$ variation of pure Hodge structure on $U$.  By
weak Lefschetz, it follows that the sheaves $R^{n-k}p_{U*}\mathbb{Q}$ are constant
for $k>0$.  In fact they are the constant sheaves given by $\rH^{n-k}(X,\mathbb{Q})$.
Then, Hard Lefschetz shows that, for $k>0$,  $R^{n+k}p_{U*}\mathbb{Q}(k)\cong
R^{n-k}p_{U*}\mathbb{Q}$.   By Deligne's degeneracy theorem~\cite{DelDeg}, $Rp_{U*}\mathbb{Q}(n)=\oplus_k R^kp_{U*}\mathbb{Q}[-k]$.  So $Rp_{U*}\mathbb{Q}(n)$ is a direct sum
of shifted constant sheaves and $\mathbf{H}[-n]$.  

The following theorem, which is Theorem C of C.~Schnell's
paper~\cite{SchRes}, shows that an analogous decomposition holds
on the level of
$Rp_*\mathbb{Q}$ provided that $\mathbf{H}$ is non-constant.
As explained by Beilinson in~\cite{BeilinsonRadon}, the result is also
a direct consequence of a much older paper of Brylinski on the Radon
transform and perverse sheaves~\cite{BrylinskiRadon}. 

\begin{theorem}[Beilinson, Brylinski, Schnell]
\label{BeiSch} Suppose $\mathbf{H}$ is non-constant.
  Then
\begin{align*}
     Rp_*\mathbb{Q}[d_{\mathcal{X}}]&=\bigoplus_k \IC(R^{n+k}p_{U*}\mathbb{Q})[-k]\\
      &=\IC(\mathbf{H})\oplus \bigoplus_{k\neq 0} \IC(R^{n+k}p_{U*}\mathbb{Q})[-k].
\end{align*}
In particular, $Rp_*\mathbb{Q}[d_{\mathcal{X}}]$ is the direct sum of
  $\IC(\mathbf{H})$ and (shifted) constant sheaves on $P^{\vee}$.
\end{theorem}

\begin{proof}[Proof of Theorem~\ref{BCor}]
  Pick $H\in P^{\vee}$ and set $Y=X\cap H$. 
We have $\IH_H^j(R^{n+i}\pi_*\mathbb{Q})=0$ for $ij\neq 0$.   So, using
Theorem~\ref{gdf} (iii), we have 
\begin{equation}
  \rH^{n+j}(Y,\mathbb{Q})=\IH^{j}_H(\mathbf{H})\oplus \IH^0(R^{n+j}\pi_*\mathbb{Q})
\end{equation}
for $j\neq 0$.
By Hard-Lefschetz, $R^{n+j}\pi_*\mathbb{Q}(j)\cong R^{n-j}\pi_*\mathbb{Q}$.
So, since $\IH^{j}_H(\mathbf{H})=0$ for $j<0$, 
 we get that $b_{n+j} Y-b_{n-j}Y =\dim \IH^j_H(\mathbf{H})$  for $j>0$ as desired.
\end{proof}

\section{Examples}\label{exs}

\subsection*{Terminology Reminder} If $S$ is an irreducible complex
scheme of finite type and $P$ is a property of closed points of $S$,
then $P$ holds for the \emph{general} point of $S$ if $P$ holds
outside of a countable union of proper subschemes
of $S$.

\begin{remark}
  I am providing this reminder because there seems to be some confusion
  in the literature about the words ``general'' and ``very general.''
  (Some authors say ``very general'' for the above notion and ``general''
  when the property holds outside of a closed proper subscheme.)
    This is the way that the term is used by
  Voisin in~\cite[p.~93]{Voisin2}, and I believe that it conforms to the usage by Laza, Sacc\`a and
  Voisin in~\cite{lsv}. 
\end{remark}

\subsection*{Cubic $4$-folds} The paper~\cite{lsv} starts with a
smooth cubic $4$-fold $X$ embedded in $P=\mathbb{P}^5$ and considers
the family $p:\mathcal{X}\to P^{\vee}$.  The family
$p_U:\mathcal{X}_U\to U$ of smooth cubic $3$ folds gives rise to a
variation of Hodge structure $\mathbf{H}_{\mathbb{Z}}$ and a family
$\pi:J(\mathbf{H}_{\mathbb{Z}})\to U$ which turns out to be a family of $5$-dimensional
abelian varieties.

\begin{theorem}[Laza---Sacc\`a---Voisin] Suppose the cubic $4$-fold $X$
  is general.   Then there is flat regular
  compactification $\bar\pi:\bar J\to P^{\vee}$ 
  with irreducible fibers. 
\end{theorem}

\begin{proof}[Explanation]  
The fact that there exists a regular flat compactification $\bar\pi:\bar J\to P^{\vee}$
is part of the main theorem of~\cite{lsv}.  
The irreducibility of the fibers is not explicitly stated
in~\cite{lsv}, but it is an important part of the construction.  Proving it 
amounts to tracing through several definitions and intermediate results 
in~\cite{lsv}, which I now do.

By~\cite[Definition 4.11]{lsv}, the compactified relative Prym variety, 
  $\cPrym \widetilde{\mathcal{C}}_B/\mathcal{C}_B$, is
    irreducible (as it is defined as an irreducible component of a larger
scheme).  By~\cite[Proposition 4.14]{lsv}, this definition is
    stable under base change so that the fiber over a point $b\in B$
    is also irreducible.  Then \cite[Proposition 5.1]{lsv} states
that $\cPrym \widetilde{\mathcal{C}}_B/\mathcal{C}_B\to
B$ is flat when $B$ is a certain Fano variety $\mathcal{F}^0$ of lines.

Section 5 of~\cite{lsv} descends the family of Prym varieties over $\mathcal{F}^0$ (which maps surjectively to $\mathbb{P}^5$)  to $P^{\vee}=\mathbb{P}^5$.  As explained in the paragraph between Lemmas 5.3 and 5.4
of~\cite{lsv}, the result is a family $\bar J\to P^{\vee}$ whose 
pullback to $\mathcal{F}^0$ is the above family of compactified Prym varieties.  Since
the compactified Prym varieties are irreducible, the fibers of $\bar J\to
P^{\vee}$ are as well.
\end{proof}

\begin{corollary}\label{PalCor2}
  Suppose $X$ is a general cubic $4$-folds and $H\in P^{\vee}$ is any
  hyperplane.  Then $Y:=X\cap H$ is palindromic.  
\end{corollary}

This motivates the following conjecture.

\begin{conjecture}\label{PalConj}
  Suppose $X$ is a  general complete intersection in $P=\mathbb{P}^N$ of 
 multi-degree $(d_1,d_2,\ldots, d_k)$ with $d_1\leq \cdots \leq d_k$.   Assume
  that $d_1\gg 0$. 
  Then, for any hyperplane  $H\in P^{\vee}$, the hyperplane section $Y:=X\cap H$
  is palindromic.
\end{conjecture}

\begin{remark}
  Perhaps ``$\gg$'' could be 
replaced with a reasonable lower bound. 
\end{remark}

Here is  a simple argument proving the conjecture when $X$
is a general surface complete intersection of multi-degree not equal to
$(2), (3)$ or $(2,2)$.  In that case, the Noether-Lefschetz theorem says
that the N\'eron-Severi group of $X$ is
$\mathbb{Z}$ with generator $[X\cap H]$ (for any hyperplane $H$).  (See Voisin's book \cite[Theorem 3.32]{Voisin2} and~\cite[Theorem 1]{Kim} for
modern proofs.)   It follows
easily that $Y:=X\cap H$ is irreducible (since $[Y]$ cannot be the direct sum
of two non-trivial effective divisors).  But, since $Y$ is a curve,
this implies that $Y$ is palindromic.

\begin{proposition}\label{wphyp}
  Suppose $X$ is a smooth, hypersurface in $P=\mathbb{P}^N$ and $H\in P^{\vee}$
  is any hyperplane section.  Then $Y:=X\cap H$ is weakly palindromic.
\end{proposition}

\begin{proof}[Sketch]
  This is well-known (see~\cite[Theorem 2.1]{DimSing}).  So I only give a sketch.  The main point is that
  $Y$ has isolated singularities.  From this one can either use the
  Clemens-Schmid exact sequence or an argument comparing the
  intersection cohomology with the ordinary cohomology.  
\end{proof}

Suppose $X$ is an arbitrary smooth cubic $4$-fold.  Since every
hyperplane section is weakly palindromic, Corollary~\ref{PalCor} does
not rule out the existence of a flat regular compactification
$\bar\pi: \bar{J}\to P^{\vee}$.  However, it can rule out the
existence of a flat regular compactification with irreducible
fibers: If $X$ is a cubic fourfold containing a non-palindromic cubic
$3$-fold $Y=X\cap H$, then the fiber $\bar\pi^{-1} \{H\}$ is not
irreducible.  To find such a cubic $4$-fold we use the following
Lemma.  

\begin{lemma}\label{Nlem}
  Suppose $Y=V(f)$ is a degree $d$ hypersurface in $\mathbb{P}^{N-1}$.
  Fix an embedding $\mathbb{P}^{N-1}\subset \mathbb{P}^{N}$.  Then there
  is a smooth degree $d$ hypersurface $X$ in $\mathbb{P}^{N}$ such
  that $Y=X\cap\mathbb{P}^N$ if and only if $Y$ has isolated singularities.
\end{lemma}

\begin{proof}[Sketch] The ``only if'' part is easy (and was already
  used above in the proof of Proposition~\ref{wphyp}).  For the ``if''
  part, suppose $f(x_1,\ldots, x_N)$ is a degree $d$ homogeneous
  polynomial.  Consider the linear subspace $V$ in
  $\rH^0(\mathbb{P}^N,\mathcal{O}_{\mathbb{P}^N}(d))$ spanned by $f$
  and $x_0h$ as $h$ runs over all degree $d-1$ homogeneous polynomials
  in the $N+1$ variables.   The base locus of the linear system $|V|$
  is $Y$.   So the general member of $|V|$ is smooth off of $Y$ by Bertini.
  But the singularities of $g=f+x_0h$ on $Y$ are contained in the intersection
  of $V(h)$ with the singularities of $Y$.  Therefore, the general member of 
  $|V|$ is smooth.  
\end{proof}

Now, I use a result of Segre and Fano as interpreted by Dolgachev.

\begin{theorem}\label{Dthm}
  There exists a cubic $3$ fold $Y$ with $10$ ordinary double
  points and $b_4 Y=6$.    
\end{theorem}

\begin{proof}
  See \cite[Proposition 1.1]{DolSeg} and the discussion shortly before
  and shortly after. 
\end{proof}

\begin{corollary}\label{cexam}
  There exists a smooth, cubic $4$-fold $X$ containing a cubic $3$-fold
  $Y$ with $b_4 Y=6$.  For such an $X$, there is no
  flat regular compactification $\bar\pi: \bar J\to P^{\vee}$
  with irreducible fibers.
\end{corollary}

\begin{proof}
  The cubic $3$-fold $Y$ with $b_4 Y=6$ is not palindromic since $b_2 Y=1$
  by weak Lefschetz.  Using Lemma~\ref{Nlem}, we can find a
  smooth cubic $4$-fold $X$ containing $Y$.  
  The result then follows from Corollary~\ref{PalCor}.
\end{proof}

\subsection*{Quadrics in Cubic $4$-folds}

Suppose $X$ is a smooth $2m$-dimensional subvariety in
$P=\mathbb{P}^N$ as in the beginning of \S\ref{Beil}. The family
$\pi:J(\mathbf{H}_{\mathbb{Z}})\to U$ of intermediate Jacobians will be an abelian
scheme provided  the Hodge structure $\rH^{2m-1}Y$ of a smooth hyperplane section $Y$ is level
$\leq 1$.   This means that $\rH^{2m-1}(Y,\mathbb{C})=F^{m} \rH^{2m-1}(Y,\mathbb{C})$. 
As in \S\ref{Beil}, we set $n=2m-1$. 

I do not have a very clear idea how often the situation above occurs for
arbitrary $X$.
However, in~\cite{RapDel}, Rapoport has a table of all complete intersections
$Y$ for which the Hodge level of the middle dimensional cohomology is $1$.
Write $V_n(d_1,\ldots, d_k)$ for the family of smooth complete intersections of dimension
$n$ coming from intersecting $k$ hypersurfaces of degrees $d_1,\ldots , d_k$ in
$\mathbb{P}^{n+k}$.  Then, according to Rapoport's table, the only
non-empty families of level $1$ with $n$ odd are:
$V_n(2,2)$, $V_n(2,2,2)$, $V_3(3)$, $V_3(2,3)$, $V_5(3)$ and $V_3(4)$.

The case where $Y$ is a cubic $3$-fold, $V_3(3)$, was the subject of
the last subsection.  In this, section I want to consider $V_3(2,3)$.

So fix a cubic $4$-fold $X$ embedded in $\mathbb{P}^5$.  Set
$\mathcal{L}:=\mathcal{O}_{\mathbb{P}^5}(1)_{|X}$.  An easy
computation shows that
$|\mathcal{L}^2|=\dim
H^0(\mathbb{P}^5,\mathcal{O}_{\mathbb{P}^5}(2))-1 =20$.   So, the
complete linear system $\mathcal{L}^2$, gives an embedding of $X$ into
$P:=\mathbb{P}^{20}$.  Cutting $X$ with hyperplanes $H\subset P$, we
get a family $p:\mathcal{X}\to P^{\vee}$ as in the beginning
of~\S\ref{Beil} which is smooth over an open subset
$U\subset P^{\vee}$.  Since the smooth hyperplane sections are
complete intersections of type $V_3(2,3)$, they have level $1$.
Therefore, the family $\pi:J(\mathbf{H}_{\mathbb{Z}})\to U$ is an abelian
scheme. In fact, Rapoport's table also gives $b_3 Y=40$
for $Y$ of type $V_3(2,3)$.  So the family $\pi:J(\mathbf{H}_{\mathbb{Z}})\to U$
is, in fact, a family of $20$-dimensional abelian varieties over a
$20$ dimensional base.

\begin{theorem}\label{fourfold}
  Let $X$ be a cubic $4$-fold as above embedded in $P=\mathbb{P}^{20}$.
  There is no regular flat compactification $\bar\pi:\bar J\to
  P^{\vee}$ of the family $\pi:J(\mathbf{H}_{\mathbb{Z}})\to U$ of intermediate
  Jacobians.
\end{theorem}

\begin{proof}
  The elements $H\in P^{\vee}$ are in 1-1 correspondence with quadrics in
  $\mathbb{P}^5$.   Pick two hyperplanes $L_1$ and $L_2$ in $\mathbb{P}^5$
  such that the cubic $3$ folds $Y_i:=X\cap L_i$ are smooth and distinct.
  Let $H$ be the point in $P^{\vee}$ corresponding to the union $L_1\cup L_2$.
  Then $Y:=X\cap H$ has two irreducible components.  Therefore $b_6 Y=2$.
  So $Y$ is not weakly palindromic.   The result follows from
  Corollary~\ref{PalCor}.
\end{proof}

\begin{remark}
  In~\cite[\S 1.3]{lsv}, Laza, Sacc\`a and Voisin point out that 
  the total space of the family $J(\mathbf{H}_{\mathbb{Z}})\to U$ admits a holomorphic symplectic form
  which would extend to any compactification $\bar\pi:\bar J\to P^{\vee}$.  They also
  show that this form is non-degenerate above a quadric if and only if the
  quadric is non-degenerate~\cite[Lemma 1.20]{lsv}.   Theorem~\ref{fourfold} can be
  seen as an amplification of this result.  
\end{remark}

\def\noopsort#1{} \def\cprime{$'$} \def\noopsort#1{} \def\cprime{$'$}

\end{document}